\documentclass[a4paper]{article}
\usepackage{a4}
\usepackage{latexsym}
\usepackage{amsmath}
\usepackage{amssymb}
\usepackage{paralist}
\usepackage{tikz}
\usepackage{tikz-cd}
\usepackage{amsthm}
\theoremstyle{plain}
    \newtheorem{theorem}                    {Theorem}       [section]
    \newtheorem{lemma}      [theorem]       {Lemma}
    \newtheorem{corollary}  [theorem]       {Corollary}
    \newtheorem{proposition}[theorem]       {Proposition}
    \newtheorem{definition} [theorem]       {Definition}

    \newtheoremstyle{alexdef}
  {.8cm}
  {.8cm}
  {\rm }
  {}
  {\bf}
  {.}
  {.5em}
  {}
\theoremstyle{alexdef}
\newtheorem{example}[theorem]{Example}

\newtheorem{remark}[theorem]{Remark}

\newcommand{\pw}{\Pi_1^{t,\ab}}

\newcommand{\Aut}{\operatorname{Aut}}
\newcommand{\pro}{{\operatorname{pro-grps}}}
\newcommand{\Hom}{\operatorname{Hom}}

\newcommand{\Ext}{\operatorname{Ext}}
\newcommand{\Gal}{\operatorname{Gal}}
\newcommand{\Pic}{\operatorname{Pic}}

\newcommand{\Spec}{\operatorname{Spec}}
\newcommand{\WS}{{\operatorname{WS}}}

\newcommand{\coker}{\operatorname{coker}}

\newcommand{\id}{\operatorname{id}}

\newcommand{\CH}{\operatorname{CH}}
\renewcommand{\lim}{\operatornamewithlimits{lim}}
\newcommand{\colim}{\operatornamewithlimits{colim}}

\newcommand{\f}{{\mathcal F}}

\newcommand{\Z}{{{\mathbb Z}}}

\newcommand{\F}{{{\mathbb F}}}

\newcommand{\Sch}{\text{\rm Sch}}
\newcommand{\et}{{\text{\rm et}}}
\newcommand{\ab}{{\text{\rm ab}}}

\newcommand{\ch}{\operatorname{char}}

\newcommand{\Cor}{\operatorname{Cor}}
\newcommand{\Frob}{\operatorname{Frob}}

\renewcommand{\div}{\operatorname{div}}

\renewcommand{\O}{{\cal O}}

\newcommand{\I}{\operatorname{I}}
\newcommand{\rec}{\operatorname{rec}}

\newcommand{\red}{\text{red}}
\newcommand{\rS}{\mathrm{S}}
\newcommand{\rW}{\mathrm{W}}

\makeatletter
\let\@fnsymbol\@arabic
\makeatother
\title{\bf Tame Class Field Theory for Singular Varieties over Finite Fields}
\author{Thomas Geisser\footnote{Supported by JSPS Grant-in-Aid (B) 23340004}\;  and Alexander Schmidt\footnote{Supported by DFG-Forschergruppe FOR 1920}}

\begin{document}
\maketitle

\section{Introduction}
In the 1980's, Kato and Saito (based on ideas of
Bloch) generalized the  class field theory for smooth, projective curves over finite fields to smooth, projective varieties  of arbitrary dimension \cite{katosaito}:
The map from the free abelian group generated by the closed points which
sends a generator $x\in X$ to the image of the Frobenius of $k(x)$
under  $\pi_1^\ab(k(x))\to \pi_1^\ab(X)$ factors through
the Chow group of zero cycles, and induces an isomorphism
\[r_X:  \CH_0(X)\stackrel{\sim}{\longrightarrow} \pi_1^\ab(X)_\rW.\]
Here $\pi_1^\ab(X)_\rW$ is the subgroup of elements in $\pi_1^\ab(X)$ whose images in the absolute Galois group of the finite base field are integral Frobenius powers.

If $X$ is not necessarily projective but still smooth, then Schmidt and Spie{\ss}
\cite{schmidtspiess,schmidt} showed that
the same result still holds if one replaces the Chow group by
Suslin homology and the fundamental group by its tame version:  The reciprocity map induces an isomorphism of
finitely generated abelian groups
\[ r_X: H_0^\rS(X,\Z) \stackrel{\sim}{\longrightarrow} \pi_1^{t,\ab}(X)_\rW.\]
This result does not extend to non-smooth schemes: the example of the node shows that $r_X$ is neither injective nor surjective  in general.
If $X$ is normal, then $r_X$ is surjective  but an example of Matsumi-Sato-Asakura \cite{mas}
shows that $r_X$ may have a nontrivial kernel.

\medskip
In this paper, we show that the result of Schmidt and Spie{\ss} can
be generalized to singular varieties if one uses a refined version of Suslin homology
on the one hand,
and replaces the fundamental group by the enlarged fundamental group of \cite[X, \S6]{sga3}
on the other hand. We denote the abelian enlarged tame fundamental group by $\Pi_1^{t,\ab}$ in order to distinguish it from the usual abelian tame fundamental group $\pi_1^{t,\ab}$, which is its profinite completion. The groups coincide if $X$ geometrically unibranch (e.g.,  normal).
Our first result is:

\begin{theorem}\label{fingen}
For any connected scheme $X$, separated and of finite type over a finite field,
the pro-group $\pw(X)_\rW $ is isomorphic to a (constant) finitely generated abelian group.
\end{theorem}

On the other hand, recall from \cite{geissersuslin} the definition of Weil-Suslin homology:
Let $\F$ be the finite base field, $F\in \Gal_\F$ be the Frobenius automorphism and $G\cong \Z$ the subgroup of $\Gal_\F$ generated by $F$.
For an abelian group $A$,
the groups $H_i^\WS(X,A)$ are defined as the homology of the cone of $1-F^*$ on the Suslin complex tensored by $A$ of the base change $\bar X$ of $X$ to the algebraic closure $\bar \F$ of $\F$.
By definition, there are   short exact sequences
\[
0\longrightarrow H_i^\rS(\bar X,A)_G \longrightarrow H_i^\WS(X,A)\longrightarrow H_{i-1}^\rS(\bar X,A)^G\longrightarrow 0.
\]
Furthermore, there are natural maps for all $i$
\[
H_i^\rS(X,A)\longrightarrow H_{i+1}^\WS(X,A).
\]
If $X$ is smooth and $A$ is finite, then it follows from the proof of Kato's conjecture by Jannsen, Kerz and Saito \cite{KS} (and under resolution of singularities) that these maps are isomorphisms
(for $i=0$ this follows from the theorem of Schmidt-Spie{\ss}).

\medskip
We define a refined reciprocity homomorphism
\[
\rec_X:  H_1^\WS(X,\Z) \longrightarrow \pw(X)_\rW
\]
such that the composite with the natural map $H_0^\rS(X,\Z)\to H_1^\WS(X,\Z)$ is the reciprocity map $r_X$ described above. Our main result, conjectured in \cite{geissersuslin}, is the following

\begin{theorem} \label{maintheorem} For any connected scheme $X$, separated and of finite type over a finite field\/ $\F$, the homomorphism
\[
\rec_X:  H_1^\WS(X,\Z) \longrightarrow \pw(X)_\rW
\]
is surjective. The kernel of\/ $\rec_X$ contains, and if resolution of singularities for schemes of dimension $\le \dim X+1$ holds over $\F$
is equal to, the maximal divisible subgroup of
$H_1^\WS(X,\Z)$.
\end{theorem}

As a corollary, we obtain (under resolution of singularities)
an isomorphism of profinite completions
\[
\rec_X^\wedge:  H_1^\WS(X,\Z)^\wedge \stackrel{\sim}{\longrightarrow} \pi_1^{t,\ab}(X).
\]
Under Parshin's conjecture (cf.\ \cite{geissersuslin}), $H_1^\WS(X,\Z)$ is finitely
generated, hence $\rec_X$ would be an isomorphism.

\section{The fundamental group and tame coverings}
\subsection{Etale and Weil-etale cohomology}

Let $X$ be a scheme, separated and of finite type over a finite field $\F$.  The absolute Galois group $\Gal_\F$
acts on $\bar X=X\times_\F \bar \F$ via its action on $\bar  \F$. Pulling an etale sheaf $\f$
on $X$ back to $\bar \f$ on $\bar X$, we obtain a sheaf with a $\Gal_\F$-action, cf.\ \cite[XIII \S 1.1]{sga7}.
Using this action, the etale cohomology $H^*_\et(X,\f)$ can be calculated
as the cohomology of $R\Gamma_{\Gal_\F} R\Gamma_\et(\bar X, \bar \f)$.
The Weil-etale cohomology of $\f$ is by definition the cohomology of
$R\Gamma_G\, R\Gamma_\et(\bar X, \bar \f)$, where $G\cong \Z$ is the subgroup
of $\Gal_\F$ generated by the Frobenius.
Since $R\Gamma_{\Gal_\F}$ and $R\Gamma_G$ coincide on discrete torsion modules, etale and Weil-etale cohomology coincide on torsion sheaves (cf.\ \cite[\S 2]{tho-weil} for a more detailed account).

\medskip
We can calculate the Weil-etale cohomology of a sheaf $\f$ as follows:
Choose an injective resolution $\f\to \I^\bullet$. Then
\[ H^i_\et(X,\f) = H^i( \I^\bullet(\bar X)^G),\]
and
\[H^i_\rW(X,A) =
H^i\big(  \I^\bullet(\bar X) \stackrel{1-F^*}{\longrightarrow}\I^\bullet(\bar X)\big),\]
where $F^*$ is the pull-back along the Frobenius, and the left complex is considered to be in homological degree $0$. We form the double complex
by using the negative of the differential in the second
complex, i.e., the differential of the total complex has the form
\begin{equation}\label{cohdiff}
\I^i(\bar X) \oplus \I^{i-1}(\bar X) \to \I^{i+1}(\bar X) \oplus \I^{i}(\bar X),\ (\alpha,\beta)\mapsto (d\alpha, \alpha-F^*\alpha-d\beta).
\end{equation}
From the definition, we obtain short exact sequences
\begin{equation}\label{hsss}
 0\longrightarrow H^{i-1}_\et(\bar X,\f)_G \longrightarrow H^i_\rW(X,\f) \longrightarrow  H^i_\et(\bar X,\f)^G\longrightarrow 0,
\end{equation}
as well as a homomorphism
\begin{equation} \label{injfor1}
H^i_\et(X,\f) \longrightarrow H^i_\rW(X,\f)
\end{equation}
for each $i\ge 0$, which is an isomorphism for $i=0$ and injective for $i=1$.

\bigskip
By \cite[Thm.\,10.2]{susvoe}, etale and qfh-cohomology of a constant sheaf $A$ coincide. Hence, in order to calculate Weil-etale cohomology of $A$, we can also work with an injective resolution $A\to \I^\bullet$ in the big qfh-site over $\F$. If moreover $A$ is a $\Z/m$-module for some $m\ge1 $, then we can also work with a resolution of $A$ by injective $h$-sheaves of $\Z/m$-modules, see \cite[Thm.\,10.7]{susvoe}.

For a regular connected curve $C$ over  a field $k$ we consider the subgroup $H^1_t(C,A)\allowbreak\subseteq H^1_\et(C,A)$ of tame cohomology classes (corresponding to those continuous homomorphisms $\pi_1^\et(C)\to A$ which factor through the tame fundamental group $\pi_1^t(C',  C' -C)$, where $C'$ is the regular compactification of $C$).

For a general $k$-scheme $X$, we call a cohomology class in $a\in H^1_\et(X,A)$ tame  if for any morphism $f:C\to X$ with $C$ a regular curve, we have $f^*(a)\in H^1_t(C,A)$. The tame cohomology classes form a subgroup
\[
H^1_t(X,A)\subseteq  H^1_\et(X,A).
\]
The groups coincide if $X$ is proper, if $p=0$, or if $p>0$ and
$A$ is $p$-torsion free, where $p$ is the characteristic of the base field $k$.
If $X$ is smooth with smooth
compactification $X'$, then $H^1_t(X,\Z/p^r) \cong H^1_\et(X', \Z/p^r)$ for any $r\ge 1$
by \cite[Prop.\,2.10]{tho-alex}.

For $X$ separated and of finite type over the finite field $\F$, we define the tame Weil-etale cohomology to be the subgroup
\[ H^1_{\rW,t}(X,A) \subseteq H^1_\rW(X, A)\]
of those elements, whose image in
$H^1_\et(\bar X,A)$ in \eqref{hsss} is tame.

\medskip
Recall that an abstract blow-up square is a cartesian diagram
\begin{equation}\label{blowupdef}
\begin{tikzcd}
Z'\rar{i'}\dar[swap]{\pi'}&X'\dar{\pi}\\
Z\rar{i}& X
\end{tikzcd}
\end{equation}
such that $i$ is a closed embedding, $\pi$ is proper, and
$\pi$ induces an isomorphism $(X'-Z')_\red \stackrel{\sim}{\to} (X-Z)_\red$.

\begin{proposition} \label{blowcohom}
If in the abstract blow-up square \eqref{blowupdef}
 $\pi$ is finite, or if the abelian group $A$ is torsion, then there is an exact sequence
\begin{multline*}
0 \longrightarrow H^0_\rW(X,A) \longrightarrow H^0_\rW(X',A) \oplus H^0_\rW(Z,A) \longrightarrow
H^0_\rW(Z',A)\\
\longrightarrow H^1_{\rW,t}(X,A)
\longrightarrow H^1_{\rW,t}(X',A) \oplus H^1_{\rW,t}(Z,A)
\longrightarrow H^1_{\rW,t}(Z',A).
\end{multline*}
\end{proposition}

\begin{proof}  Let $\Sch/\F$ be the category of separated schemes of finite type over $\F$.
For $S\in \Sch/\F$ we denote by $\Z_h(S)$ the $h$-sheaf of abelian groups associated with the presheaf defined by $U \mapsto \Z[\mathrm{Mor}_\F(U,S)]$.

For any finite field extension $\F'/\F$, the base changes of $X$, $Z$, $Z'$, $X'$ to $\F'$ form an abstract blow-up square in a natural way.  By the same argument as in the proof of \cite[Prop.\,3.2]{ge-duke} (for the eh-topology) or \cite[Lem.\,12.1]{blowup} (for the cdh-topology), we have an exact sequence of $h$-sheaves on $\Sch/\F$
\begin{equation}\label{htria}
0\to \Z_h(Z'_{\F'}) \to \Z_h(Z_{\F'}) \oplus \Z_h(X'_{\F'}) \to \Z_h(X_{\F'}) \to 0.
\end{equation}
If $A$ is torsion, \'{e}tale and $h$-cohomology with values in $A$ agree by \cite[Thm.\,10.2]{susvoe}. Applying the functor $R\Hom_h(-,A)$ and passing to the limit over all $\F'/\F$,  we obtain  the exact triangle
\begin{equation} \label{eqtriang}
 R\Gamma(\bar X_\et, A) \to R\Gamma(\bar X'_\et,A) \oplus R\Gamma(\bar Z_\et, A) \to R\Gamma(\bar Z'_\et, A)\to  R\Gamma(\bar X_\et, A)[1]
\end{equation}
and the long exact sequence
\begin{equation} \label{eqoben}
\to H^i_\et(\bar X,A) \to H^i_\et(\bar X',A) \oplus H^i_\et(\bar Z,A) \to H^i_\et(\bar Z',A) \to  H^{i+1}_\et(\bar X,A) \to \ .  \end{equation}
If $\pi$ is finite, we have the qfh-version of the exact sequence \eqref{htria} and obtain \eqref{eqtriang} and \eqref{eqoben} for arbitrary $A$, since etale and qfh-cohomology with values in any abelian group agree by \cite[Thm.\,10.7]{susvoe}.
Applying $R\Gamma_G$ to (\ref{eqtriang}) and taking cohomology, we obtain the long exact sequence
\begin{equation}\label{eqweil}
H^i_\rW(X,A) \to H^i_\rW(X',A) \oplus H^i_\rW(Z,A) \to H^i_\rW(Z',A) \to  H^{i+1}_\rW(X,A) \to  \cdots .
\end{equation}
By \cite[Prop.\,5.1]{tho-alex}, (\ref{eqoben}) induces an exact sequence
\begin{multline}\label{tamelong}
0 \longrightarrow H^0_\et(\bar X,A) \longrightarrow H^0_\et(\bar X',A) \oplus H^0_\et(\bar Z,A) \longrightarrow H^0_\et(\bar Z',A)\\
\stackrel{\delta}{\longrightarrow} H^1_t(\bar X,A) \longrightarrow H^1_t(\bar X',A) \oplus H^1_t(\bar Z,A) \longrightarrow H^1_t(\bar Z',A).
\end{multline}
Comparing \eqref{tamelong} with the sequences (\ref{eqoben}) and (\ref{eqweil}),  we obtain the statement of the proposition by a diagram chase.
\end{proof}

\subsection{The enlarged fundamental group}
We recall the definition of the enlarged fundamental group
of \cite[X \S 6]{sga3}: Let $X$ be a connected, locally noetherian scheme.
For a group $G$ (considered as a constant group scheme over $X$), a $G$-torsor $P$ over $X$
is a non-empty etale $X$-scheme $P$ (i.e., $\pi: P\to X$ is unramified, flat and locally of finite type) with
a $G$-action $ P\times_X G \to P$ such that   $P\times_X G\to P\times_XP$,
$(x,g)\mapsto (x,xg)$, is an isomorphism.
By \cite[Prop.\,2.7]{milnebook} (see also Ex.\ 2.6, loc.\ cit.), for any etale sheaf $F$ on $X$ we have a Hochschild-Serre spectral sequence
\[
E_2^{rs}= H^r(G, H^s_\et(P,\pi^*F)) \Longrightarrow H^{r+s}_\et(X,F).
\]
For a geometric point $\xi$ of $X$, one defines $\Pi^1(X,\xi,G)$ to
be the set of isomorphism classes of $G$-torsors over $X $ pointed over $\xi$. The trivial $G$-torsor gives a distinguished element in $\Pi^1(X,\xi,G)$.
For a $G$-torsor $P$ on $X$ and a group homomorphism $ f:G\to H $, consider
\[
f_*(P):= (P\times H)/G,
\]
where $G$ acts by $(t,h)\cdot g= (tg,f(g^{-1})h)$. Then $f_*(P)$ is an $H$-torsor over $X$ and we obtain a functor
\[ G \longmapsto \Pi^1(X,\xi,G)\]
from groups to pointed sets.
By \cite[X \S 6]{sga3},
this functor is pro-represented by the enlarged fundamental pro-group
$\Pi_1(X,\xi)$, i.e.,
\[
\Pi^1(X,\xi,G)\cong \Hom_\pro(\Pi_1(X,\xi), G).
\]
Explicitly, there is a pro-system of groups
$\Pi_1(X,\xi)=(G_i)_{i\in I}$ with $I$ filtering,
and a $G_i$-torsor $P_i$ corresponding to the projection
map $\Pi_1(X,\xi)\to G_i $ for all $i$ such that for any transition map
$\alpha_{ij}:G_i\to G_j$ in the system we have $P_j=(\alpha_{ij})_*(P_i)$,
and such that for any morphism
$\Pi_1(X,\xi)\to H$ represented by $f:G_j\to H$ in the filtered
colimit, the corresponding $H$-torsor is
$f_*(P_j)$.

\medskip
Next we define the enlarged tame fundamental group by extending the notion of curve-tameness from \cite{kerzschmidt} to the enlarged context.
A $G$-torsor $P$ over a regular connected curve $C$ over a field $k$ is called tame, if the projection $P\to C$ extends to an at most tamely ramified covering of the regular compactification $C'$ of $C$. If $X$ is any scheme, separated and of finite type over $k$, and $P$ is a $G$-torsor on $X$, then we call $P$ tame if its pull back to the normalization of any curve on $X$ is a tame torsor. If $G=A$ is abelian, then an $A$-torsor is tame if and only if its associated class in $H^1_\et(X,A)$ lies in $H^1_t(X,A)$.

\medskip
The functor
\[
G \longmapsto \Pi^{1,t}(X,\xi,G)
\]
sending $G$ to the set of isomorphism classes of pointed tame $G$-torsors on $X$
is pro-represented by the enlarged tame fundamental group $\Pi_1^{t}(X,\xi)$, a quotient of $\Pi_1(X,\xi)$ in the category of pro-groups.

The abelianizations  $\Pi_1^{\ab}(X)$ and $\Pi_1^{t,\ab}(X)$ of $\Pi_1(X,\xi)$ and $\Pi_1^t(X,\xi)$ represent the restrictions of the respective functors to the category of abelian groups and are independent of the chosen base point $\xi$.

\begin{lemma}\label{proab}
For any abelian group $A$ we have
\[ \Hom_\pro(\Pi_1^{\ab}(X),A)\cong H^1_\et(X,A)\]
and similarly for the tame version.
\end{lemma}

\begin{proof}
Write $\Pi_1^\ab(X)=(A_i)$ and $P_i$ for the $A_i$-torsor
corresponding to the projection
$\Pi_1^\ab(X)\to A_i$.
We obtain a filtered direct system of
Hochschild-Serre spectral sequences
\[
H^r(A_i,H^s_\et(P_i,A))\Rightarrow H^{r+s}_\et(X,A)
\]
inducing a system of short exact sequences
\[
0\longrightarrow H^1(A_i,A)\longrightarrow H^1_\et(X,A)\longrightarrow
H^0(A_i,H^1_\et(P_i,A)).
\]
The right map is the zero map in the colimit over all $i$, because
if the $A$-torsor $P$ arises from a map $f:A_i\to A$,
i.e., $P=f_*(P_i)$, then $P$ trivializes over $P_i$.
Finally,
\[
\colim H^1(A_i,A)\cong \colim \Hom(A_i,A)=
\Hom_\pro(\Pi_1^{\ab}(X),A). \qedhere
\]
\end{proof}

From now on let $X$ be a connected scheme, separated and of finite type over the finite field $\F$. As above, we denote the subgroup of $\Gal_\F$ consisting of integral powers of the Frobenius automorphism by $G$.

\begin{definition}
The enlarged Weil-tame fundamental group $\Pi_1^t(X,\xi)_\rW$ is defined by the cartesian diagram  of pro-groups
\[
\begin{tikzcd}
\Pi_1^t(X,\xi)_\rW \rar\dar &\Pi_1^t(X,\xi)\dar\\
G \rar & \Pi_1(\F).
\end{tikzcd}
\]
\end{definition}

The abelianization $\Pi_1^{t,\ab}(X)_\rW$ of $\Pi_1^{t}(X,\xi)_\rW$ fits into the analogous
cartesian diagram
\[
\begin{tikzcd}
\Pi_1^{t,\ab}(X)_\rW \rar\dar &\Pi_1^{t,\ab}(X)\dar\\
G \rar & \Pi_1(\F).
\end{tikzcd}
\]
The profinite completion of $\pw(X)_\rW$  is the usual
abelian (curve-)tame fundamental group $\pi_1^{t,ab}(X)$ of  \cite{kerzschmidt}.

\begin{proposition}\label{homvergleich}
For any abelian group $A$, there is a functorial isomorphism
\[
\Hom_\pro(\Pi_1^{\ab}(X)_\rW,A)\cong H^1_{\rW}(X,A)
\]
compatible with the isomorphism of Lemma \ref{proab}, and
there is a similar isomorphism for the tame version.
\end{proposition}

\begin{proof}
Replacing $\F$ by its maximal algebraic extension in $\Gamma(X,\O_X)$ changes $G$ to a subgroup of finite index, but does not change the groups on both sides of the statement. Hence we may assume that $X$ is geometrically connected \cite[I, \S4, 6.7]{dg}.
Setting $\bar X=X\times_\F \bar \F$, we have the exact sequence of pro-groups
\[
1 \longrightarrow \Pi_1(\bar X,\xi) \longrightarrow \Pi_1(X,\xi)_\rW \longrightarrow G \longrightarrow 1.
\]
If we write $\Pi_1(X,\xi)_\rW=(G_i)$, then $\Pi_1(\bar X,\xi)=(\bar G_i)$ with
$\bar G_i=\ker (G_i\to G)$. We denote the $\bar G_i$-torsor over $\bar X$ associated with the projection map $\Pi_1(\bar X,\xi)\to \bar G_i$
by $\bar P_i$.

Consider the functor $\f\mapsto \Gamma(\bar P_i, \f)$ from the category of etale sheaves on $X$ to $G_i$-modules.
The inclusion
$A\to R\Gamma(\bar P_i,A)$ induces a map $R\Gamma_{G_i}(A)\to R\Gamma_{G_i}R\Gamma(\bar P_i,A)$ in the derived category of abelian groups. Since $R\Gamma_{G_i}=R\Gamma_G R\Gamma_{\bar G_i}$, we can write this map in the form
\begin{equation}\label{eq5}
R\Gamma_{G_i} (A)
\longrightarrow R\Gamma_GR\Gamma_{\bar G_i}R\Gamma(\bar P_i,A) .
\end{equation}
Since taking global sections over $\bar P_i$ has an exact left
adjoint, it sends
injectives sheaves on $X$  to injective $\Z \bar G_i$-modules, and
we obtain
$R\Gamma_{\bar G_i}R\Gamma(\bar P_i,A) = R\Gamma(\bar X,A)$. We thus can write (\ref{eq5}) in the form
\begin{equation}\label{eq6}
R\Gamma_{G_i} (A) \longrightarrow R\Gamma_G R\Gamma(\bar X,A).
\end{equation}
Taking cohomology, and passing to the colimit over $i$, we obtain  maps
\[
H^n(\Pi_1^{\ab}(X)_\rW,A)\longrightarrow H^n_{\rW}(X,A), \quad n\ge 0.
\]
For $n=1$ this is the map of the proposition and we have to show that it is an isomorphism. For this we rewrite (\ref{eq6}) in the form
\begin{equation}\label{eq7}
R\Gamma_G R\Gamma_{\bar G_i} (A) \longrightarrow R\Gamma_G R\Gamma(\bar X,A)
\end{equation}
and consider the map of associated spectral sequences, which degenerate to short exact sequences. In degree $1$, we obtain the commutative diagram with exact lines
\[
\begin{tikzcd}[column sep=small]
0 \rar & A_G \dar\rar & H^1(\Pi_1(X)_\rW,A)\dar\rar & H^1(\Pi_1(\bar X),A)^G\dar\rar &0\phantom{.}\\
0\rar& H^0_\et(\bar X,A)_G \rar & H^1_\rW(X,A)\rar&H^1_\et(\bar X,A)^G \rar & 0.
\end{tikzcd}
\]
Since $X$ is geometrically connected, the left hand vertical map is an isomorphism.  The right hand vertical map is an isomorphism by Lemma~\ref{proab}.
Hence the middle map is  an isomorphism.

To show the statement for the tame variant, we note that there is a similar diagram as above for the tame groups. Indeed,  a torsor on $X$ is tame if and only if its base change to $\bar X$ is tame, so we obtain the exact sequence
\[
1 \longrightarrow \Pi_1^t(\bar X,\xi) \longrightarrow \Pi_1^t(X,\xi)_\rW \longrightarrow G \longrightarrow 1.
\]
On the other hand, the lower sequence of the above diagram induces a similar sequence for the  tame Weil-etale cohomology by definition. This time, the right hand vertical map is an isomorphism by the tame version of Lemma~\ref{proab}.
\end{proof}

\medskip
For a finite disjoint union of connected schemes $X=\amalg  X_i$  we write by abuse of notation
\[
\Pi_1^{t,\ab}(X)_\rW= \prod_i \Pi_1^{t,\ab}(X_i)_\rW.
\]

\begin{theorem}
For any $X$, separated and of finite type over a finite field\/ $\F$,
$\pw(X)_\rW $ is isomorphic to a finitely generated abelian group.
\end{theorem}

\begin{proof}
Let us first assume that $X$ is normal and connected.  We claim that  the kernel
\[
\Pi_1^{t,\ab}(X)^\mathrm{geo}:= \ker\big(\Pi_1^{t,\ab}(X)_\rW \longrightarrow \Pi_1(\F)_\rW \cong \Z\big)
\]
is a finite abelian group.  If $X$ is smooth, this follows from the main theorem of Schmidt-Spie{\ss} \cite{schmidtspiess,schmidt}.
For a general normal $X$, choose a dense open smooth subscheme $U\subset X$. Then, by \cite[V, Prop.\,8]{sga1},  $\Pi_1^{t,\ab}(U)^\mathrm{geo}$ surjects onto $\Pi_1^{t,\ab}(X)^\mathrm{geo}$, hence the latter group is finite.

Now let $X$ be arbitrary. We can assume that $X$ is connected and reduced. Let $X' \to X$ be the normalization. The cokernel $\Pi_1^\ab(X'/X)$ of $\Pi_1^\ab(X')\to \Pi_1^\ab(X)$ represents the functor $\Pi^1_\ab(X'/X)$ which sends an abelian group $A$ to the set of isomorphism classes of $A$-torsors on $X$ which trivialize over $X'$. We denote the tame version of this group by $\Pi_1^{t,\ab}(X'/X)$ and the cokernel of $\Pi_1^{t,\ab}(X')_\rW\to \Pi_1^{t,\ab}(X)_\rW$ by $C$.   Consider the diagram
\[
\begin{tikzcd}
\Pi_1^\ab(X')\rar \dar & \Pi_1^\ab(X)\dar\rar & \Pi_1^\ab(X'/X) \dar{\alpha}\rar & 0\\
\Pi_1^{t,\ab}(X')\rar  & \Pi_1^{t,\ab}(X)\rar & \Pi_1^{t,\ab}(X'/X) \rar & 0\\
\Pi_1^{t,\ab}(X')_\rW\rar \uar & \Pi_1^{t,\ab}(X)_\rW\uar\rar & C \uar[swap]{\beta}\rar & 0&.
\end{tikzcd}
\]
Since $X'\to X$ is proper, a torsor on $X$ which trivializes over $X'$ is  tame.  Hence $\alpha$ is an isomorphism, and so is $\beta$. By  \cite[X \S 6, p.\,109]{sga3}, $\Pi_1^\ab(X'/X)$ is a finitely generated abelian group, hence so is~$C$. We proved above that the geometric part of $\Pi_1^{t,\ab}(X')_\rW$ (defined componentwise if $X'$ is not connected) is finite. This implies that $\Pi_1^{t,\ab}(X)_\rW$ is constant and finitely generated.
\end{proof}

\section{Weil-Suslin homology}
Let $k$ be a perfect field and $X$ a scheme, separated and of finite type over~$k$. We recall that, for a smooth $k$-scheme $T$,
the group of finite correspondences $\Cor(T,X)$ is the free abelian group
generated by closed integral $Z\subseteq T\times X$ which
are finite and surjective over a component of $T$.
The Suslin complex of $X$ is the complex $C_\bullet(X)=\Cor(\Delta^\bullet,X)$, where $\Delta^i=\Spec k[T_0,\ldots,T_i]/(\Sigma T_i-1)$. Putting $\partial := \sum_{j=0}^i (-1)^j \delta^j_i\in \Cor(\Delta^{i-1},\Delta^i)$, where $\delta_i^j:\Delta^{i-1}\to \Delta^i$, $j=0,\ldots,i$, are the face maps, the differential
$\Cor(\Delta^i,X)\to \Cor(\Delta^{i-1},X) $ is given as the composition of correspondences
$x\mapsto x\circ \partial$.
The following lemma is easy to check from the definitions:

\begin{lemma}\label{corlemma}
Let $f:X\to Y$ be a morphism of schemes. \medskip

{\rm a)} If\/ $X$ and $T$ are smooth and $c\in \Cor(T,X)$, then $(\id_T\times f)_*c= f\circ c$.
Here the left hand side is push-forward of cycles, and the right hand side
is composition of correspondences.

{\rm b)} If\/ $X$ and $Y$ are smooth and $d\in \Cor(Y,Z)$, then $(f\times \id_Z)^*d= d\circ f$.
Here the left hand side is pull-back of cycles, and the right hand side
is composition of correspondences.

{\rm c)} If\/ $f$ is an automorphism of the smooth scheme $X$,
then $f^*c= f^{-1}_*c$ for any cycle~$c$.
\end{lemma}

Let $T$ and $X$ be separated schemes of finite type over $k$ and $\sigma \in \Gal(\bar k/k)$.
Then $\sigma$ acts on $\bar X=X\times_k\bar k$ via its action on $\bar k$,
and  on algebraic cycles by pull-back.

\begin{lemma} \label{galact}
The action of $\sigma$ on $\Cor_{\bar k}(\bar T,\bar X)$ induced
by pull-back of algebraic cycles sends a correspondence $\alpha$ to
the composition $\sigma_X^{-1} \, \alpha \,\sigma_T$, where $\sigma_X$ and $\sigma_T$ are the automorphisms of $\bar X$ and $\bar T$ induced by $\sigma$. In other words, the following diagram of correspondences commutes:
\begin{equation}\label{coract}
\begin{tikzcd}
\bar T\rar{\sigma_T}\dar[swap]{\sigma^*\alpha} & \bar T\dar{\alpha}\\
\bar X\rar{\sigma_X}&\bar X.
\end{tikzcd}
\end{equation}
\end{lemma}

\begin{proof}
From  Lemma~\ref{corlemma}, we have
\begin{multline*}
\sigma_X \circ \sigma^*\alpha=
(\id_T\times \sigma_X)_* (\sigma_T\times \sigma_X)^*\alpha\\
= (\id_T\times \sigma_X^{-1})^* (\id_T\times \sigma_X)^*(\sigma_T\times \id_X)^* \alpha
 = \alpha\circ \sigma_T.
\end{multline*}
\end{proof}

Now we assume that $k=\F$ is a finite field.  Let $\bar X$ be the extension to
the algebraic closure $\bar \F$, and let $F$ be the Frobenius automorphism of $\bar \F/\F$.
Let $G\cong \Z$ be the Weil group of $\F$, generated by the Frobenius $F$.

\begin{definition}
Weil-Suslin homology $H_i^\WS(X,A)$ with coefficients in the abelian group $A$
is defined as the homology of the cone of
\[C_\bullet(\bar X)\otimes A \stackrel{1-F^*}{\longrightarrow} C_\bullet(\bar X)\otimes A,\]
i.e., the total complex of the double complex
\[
\begin{tikzcd}
\cdots \rar&C_2(\bar X)\otimes A \rar{-\partial}\dar{1-F^*} &C_1(\bar X)\otimes A\rar{-\partial}\dar{1-F^*}&C_0(\bar X)\otimes A\dar{1-F^*}\\
\cdots \rar&C_2(\bar X)\otimes A \rar{\partial} &C_1(\bar X)\otimes A\rar{\partial}&C_0(\bar X)\otimes A.
\end{tikzcd}
\]
\end{definition}

In degree $i$, the total complex consists of elements
\[(x_i,x_{i-1})\in
\Cor(\bar \Delta^i,\bar X)\otimes A \, \oplus \, \Cor(\bar \Delta^{i-1},\bar X)\otimes A\]
with differential
\begin{equation}\label{part0}
(x,y) \longmapsto
( x\partial + y- F^{-1}yF ,- y\partial).
\end{equation}
The spectral sequence for double complexes gives short exact sequences
\begin{equation}\label{hsss1}
0\to H_i^\rS(\bar X,A)_G \to H_i^\WS(X,A) \to H_{i-1}^\rS(\bar X,A)^G\to 0
\end{equation}
where the left hand side and right hand side are the coinvariants and invariants
with respect to $G$, respectively.
The map $C_\bullet(X) \to C_\bullet(\bar X)$, sending a generator
$Z\subseteq X\times \Delta^i$ to its pull-back
to the algebraic closure, has image in the kernel of $1-F^*$. Therefore, we obtain  natural maps  for $i\ge 0$.
\begin{equation}\label{sustows}
H_i^\rS(X,A)\longrightarrow H_{i+1}^\WS(X,A).
\end{equation}
\begin{remark}\label{wsalsabglfunktr}
For a torsion $\Gal_\F$-module $M$, we have $R\Gamma(\F,M)\cong [M \xrightarrow{1-F^*} M]$, where the last complex is concentrated in (cohomological) degrees zero and one. Hence, if $A$ is a torsion group, then
\[
H_i^\WS(X,A)= H^{1-i}\Big(R\Gamma\big(\F,C_\bullet(\bar X)\otimes A \big) \Big).
\]
\end{remark}
\begin{remark}\label{basefieldchange}
The definition of Weil-Suslin homology depends on the finite base field $\F$ (via $F\in \Gal_\F$). However, if $X\to \F$ factors through $\F'/\F$, then the Weil-Suslin homology of $X$ does not depend on whether we consider $X$ as a scheme over $\F'$ or over $\F$.
\end{remark}

\begin{proposition}\label{blowupseq}
An abstract blow-up diagram \eqref{blowupdef}
induces a long exact sequence of Weil-Suslin homology groups
\begin{multline*}
H_1^\WS(Z',A) \longrightarrow  H_1^\WS(X',A) \oplus H_1^\WS(Z,A) \longrightarrow
H_1^\WS(X,A)\longrightarrow  \\
H_0^\WS(Z',A) \longrightarrow H_0^\WS(X',A) \oplus H_0^\WS(Z,A) \longrightarrow
H_0^\WS(X,A) \longrightarrow 0.
\end{multline*}
\end{proposition}

\begin{proof}
By definition of Weil-Suslin homology, only the terms $C_i(\bar ?)$ for $i\leq 2$
are involved in the definition of the terms in the sequence,
hence a diagram chase shows that it suffices to show that in
the short exact sequence of complexes
\[ 0\to C_\bullet(\bar Z') \to C_\bullet(\bar X')\oplus C_\bullet(\bar Z)
\to C_\bullet(\bar X)  \to K_\bullet \to 0\]
one has $H_i(K_\bullet)=0$ for $i\leq 2$.
This was shown in the proof of  \cite[Prop.\,5.2]{tho-alex}.
\end{proof}

Since $H_0^\rS(\bar \F,\Z) = \Z$ and $H_i^\rS(\bar \F,\Z) = 0$ for $i\geq 1$, we can calculate the Weil-Suslin homology of the point using \eqref{hsss1} as follows.
\begin{example}\label{wsofpoint}
\[
H_i^\WS(\F,\Z) \cong \left\{
\begin{array}{ll}
\Z& i=0,1,\\
0 & i\ge 2.
\end{array} \right.
\]
\end{example}
Let $C$ by a smooth, proper, geometrically connected curve over $\F$. By \cite{Li}, we have
\[
H_i^\rS(\bar C,\Z) \cong \left\{
\begin{array}{ll}
\Pic(\bar C)& i=0,\\
\bar \F^\times & i= 1,\\
0 & i \ge 2.
\end{array} \right.
\]
Regarding $\bar \F^\times_G = 0$ and $\Pic(\bar C)_G\cong \Z$,  the exact sequence \eqref{hsss1} yields the following:

\begin{example}\label{wssmothcurve}
Let $C$ be a smooth, proper, geometrically connected curve over\/ $\F$. Then
\[
H_i^\WS(C,\Z) \cong \left\{
\begin{array}{ll}
\Z& i=0,\\
\Pic(C) & i= 1,\\
\F^\times& i= 2,\\
0& i\ge 3.
\end{array} \right.
\]
\end{example}

The following improves \cite[Prop.\,7.8]{geissersuslin}:

\begin{proposition}\label{sushomdiv}
Let $X$ be a connected, separated scheme of finite type over $\F$.  Then the structure map induces an isomorphism
\[
\deg: H_0^\WS(X,\Z)\stackrel{\sim}{\longrightarrow} H_0^\WS(\F,\Z)\cong \Z .
\]
\end{proposition}

\begin{proof} We have $H_0^\WS(X,\Z)\cong H_0^\rS(\bar X,\Z)_G$, hence $\deg$ is surjective and it remains to show that its kernel is trivial.
Since elements in $H_0^\rS(\bar X,\Z)$ are represented by zero-cycles,  any element of $H_0^\WS(X,\Z)$ comes by push-forward from $H_0^\WS(C,\Z)$ for some connected curve $C\subset X$ (use, e.g., \cite{mumford}, II \S 6 Lemma). We therefore can  assume that $X=C$ is a connected curve.

If $C'\to C$ is finite and surjective, then
$
 H_0^\WS(C',\Z)\to H_0^\WS(C,\Z)
$
is surjective. Moreover, any element of degree zero in $H_0^\WS(\bar C,\Z)$ can be lifted to an element in the kernel of the multi-degree map
\[
 H_0^\WS(C',\Z) \longrightarrow \Z^{\pi_0(C')}.
\]
We therefore can assume that $C$ is a normal, connected curve. Moreover, we can assume that $C$ has an $\F$-rational point (use Remark~\ref{basefieldchange}).
Let $\mathcal{A}$ be the semi-abelian Albanese variety of $C$. Then, by \cite[Thm.\,3.1]{susvoe} (and \cite{Li} if $C$ is proper), the degree zero part of $H_0^\rS(\bar C,\Z)$ is isomorphic to $\mathcal{A}(\bar \F)$. The $G$-coinvariants of this group are both finite and divisible, hence trivial.
\end{proof}

\begin{corollary} \label{sushomdivcor}
The canonical injection
\[
H_1^\WS(X,\Z)/m \longrightarrow H_1^\WS(X,\Z/m)
\]
is an isomorphism for any $m\ge1 $.
\end{corollary}

\begin{proof}
The cokernel is isomorphic to $\ker(H_0^\WS(X,\Z)\stackrel{\cdot m}{\to} H_0^\WS(X,\Z))=0$.
\end{proof}

\section{Duality}

We say that ``resolution of singularities holds for schemes of dimension $\leq d$ over a perfect field $k$'' if the following two conditions are satisfied.

\medskip
\begin{compactitem}
\item[\rm (1)] For any integral separated scheme of finite type $X$ of dimension $\leq d$ over~$k$,  there exists a
projective birational morphism $Y \to X$ with $Y$ smooth over~$k$ which is an isomorphism over the regular locus of $X$.
\item[\rm (2)] For any integral smooth scheme $X$ of dimension $\leq d$ over $k$ and any birational proper morphism
$Y \to X$ there exists a tower of morphisms $X_n \to X_{n-1} \to \dots \to X_0= X$, such that $X_n \to X_{n-1}$  is a blow-up with a smooth center for $i=1,\ldots,n$, and such that the composite morphism $X_n \to X$ factors through $Y \to X$.
\end{compactitem}

\bigskip
In this paper, we do not use the full duality statement below, but only the equality of the orders of the respective groups.

\begin{theorem}\label{duall} Let $X$ be separated and of finite type over the finite field\/ $\F$ of characteristic~$p$.
If $m$ is prime to $p$, then
there is a perfect pairing of finite groups
\[H_i^{\WS}(X,\Z/m)\times H^i_\et(X,\Z/m)\longrightarrow \Z/m.\]
If\/ $X$ is smooth and resolution of singularities for schemes of dimension $\le \dim X+1$ holds over $\F$,
then there is a perfect pairing of finite groups for any $r\ge 1$,
\[H_1^{\WS}(X,\Z/p^r)\times H^1_t(X,\Z/p^r)\longrightarrow \Z/p^r.\]
\end{theorem}

\begin{proof}
By \cite[Thm.\,5.4, Thm.\,5.5]{geissersuslin}, we have a perfect pairing of finite groups
\[
H_{i-1}^{\mathrm{GS}}(X,\Z/m) \times H^i_{\et}(X,\Z/m) \longrightarrow \Z/m,
\]
where $H_{i-1}^{\mathrm{GS}}$ is Galois-Suslin homology. By \cite[\S 7.1]{geissersuslin}, we have
$H_{i-1}^{\mathrm{GS}}(X,\Z/m)\cong H_{i}^{\WS}(X,\Z/m)$, showing the  first statement note that $H_{i}^{\WS}$ is denoted by $H_{i}^{ar}$ in \cite{geissersuslin}).

For the second statement, let $X'$ be a smooth, proper variety containing $X$ as a dense open subscheme.
Then, by \cite[Prop.\,6.2]{tho-alex}, we have
$
H_i^\rS(\bar X,\Z/p^r)\cong H_i^\rS(\bar X',\Z/p^r)
$
for $i=0,1$, and the exact sequence \eqref{hsss1} above implies that
\[
H_1^{\WS}(X,\Z/p^r )\cong H_1^{\WS}(X',\Z/p^r).
\]
Furthermore,
$H^1_t(X,\Z/p^r)\cong H^1_\et(X',\Z/p^r)$ by \cite[Prop.\,2.10]{tho-alex}.
Hence we may assume that $X=X'$  is smooth and proper.

Let $\Z^c_X$ be the complex of \'{e}tale sheaves on $X$ which associates to $U\to X$  the Bloch complex $z_0(U,\bullet)$. Then, for smooth proper $X$, $H^i_\et(X,\Z/p^r)$ is dual
to $\Ext^{2-i}_X(\Z/p^r, \Z^c_X)$ by \cite[Thm.\,5.1]{tho-duality}. Furthermore,
\[
\renewcommand{\arraystretch}{1.5}
\begin{array}{rcll}
\Ext^{2-i}_X(\Z/p^r, \Z^c_X)&\cong& \Ext^{1-i}_{X,\Z/p^r}(\Z/p^r,\Z^c_X/p^r) & \text{ (by \cite[Lem.\,2.4]{tho-duality}})\\
&\cong& H^{1-i}_\et(X,\Z^c_X/p^r)\\
&\cong & H^{1-i}\big(R\Gamma(\F,R\Gamma(\bar X_\et,\Z^c_{\bar X}/p^r))\big)\\
&\cong & H^{1-i}\big(R\Gamma(\F,\Z^c_{\bar X}/p^r(\bar X))\big) & \text{ (by \cite[Thm.\,3.1]{tho-duality})}.
\end{array}
\renewcommand{\arraystretch}{1}
\]
The natural map
\[
C_\bullet(\bar X) \otimes \Z/p^r \longrightarrow \Z^c_{\bar X}/p^r(\bar X)
\]
induces an isomorphism on $H^j$ for $j=-1,0$ by \cite[Thm.\,2.7]{schmidtspiess} (the assumption $\dim X \leq 2$ is unnecessary and not used in the proof). Hence for $i=1$ we obtain
\[
\renewcommand{\arraystretch}{1.5}
\begin{array}{rcll}
H^0\big(R\Gamma(\F,\Z^c_{\bar X}/p^r(\bar X))\big) &\cong &H^0\big(R\Gamma(\F,C_\bullet(\bar X)\otimes \Z/p^r)\big)\\
&\cong& H_1^\WS(X,\Z/p^r) &\text{(by Remark~\ref{wsalsabglfunktr})}.
\end{array}
\renewcommand{\arraystretch}{1}
\]
This concludes the proof.
\end{proof}

\section{The reciprocity map}
For any $X$ and any abelian group $A$, we construct a functorial pairing
\begin{equation}\label{pairing}
H_1^\WS(X,\Z) \ \times \ H^1_{\rW,t}(X,A) \stackrel{\langle \ , \ \rangle}{\longrightarrow} A,
\end{equation}
which induces natural maps
\[ \Hom(\pw(X)_\rW,A)\cong H^1_{\rW,t}(X,A)  \longrightarrow
\Hom(H_1^\WS(X,\Z),A) \]
for any abelian (pro-)group $A$, hence the Yoneda lemma induces
\begin{equation}\label{rec}
\rec_X:  H_1^\WS(X,\Z) \longrightarrow  \pw(X)_\rW .
\end{equation}

The pairing \eqref{pairing} should satisfy two conditions.  First of all,  the composite of\/ $\rec_X$ with the natural map $H_0^\rS(X,\Z) \to H_1^\WS(X,\Z)$ should be the map
\[
r_X: H_0^\rS(X,\Z)  \longrightarrow \Pi_1^{t,\ab}(X)_\rW
\]
which sends the class $[x]$ of a closed point $x\in X$ to its Frobenius automorphism in $\Pi_1^{t,\ab}(X)_\rW$.

Secondly, we want \eqref{pairing} to be compatible with the pairing over $\bar \F$ considered in \cite{tho-alex}:
 in loc.\ cit.\ we considered (over any algebraically closed field) a natural pairing
\begin{equation}\label{paarungoben}
H_1^\rS(\bar X,\Z)\ \times \ H^1_t(\bar X,A )\longrightarrow A,
\end{equation}
which is defined by pulling back torsors along finite correspondences. For $A=\Z/m$, $(m,\ch(\bar \F))=1$, the induced
homomorphism
\[
\Hom(H_1^\rS(\bar X,\Z),\Z/m) \longrightarrow H^1_\et(\bar X, \Z/m)
\]
coincides with the composite of
$\Hom(H_1^\rS(\bar X,\Z),\Z/m)\stackrel{can}{\hookrightarrow} H^1_\rS( \bar X,\Z/m)$ with the Suslin-Voevodsky comparison isomorphism (\cite[Cor.\,7.8]{susvoe})
\[
H^1_\rS( \bar X,\Z/m) \xrightarrow[\sim]{SV} H^1_\et(\bar X, \Z/m).
\]
We want to construct  the pairing \eqref{pairing} in such a way that the diagram
\[
\begin{tikzcd}[column sep=tiny]
H_1^\rS(\bar X,\Z)\dar &{\times} & H^1_{t}(\bar X,A)\ \arrow{rrrrr}{\eqref{paarungoben}} &&&&&\phantom{}&A\dar[equal]\\
H_1^\WS(X,\Z) &{\times} & H^1_{t}(X,A)\ \uar \arrow{rrrrr}{\eqref{pairing}} &&&&&\phantom{}& A
\end{tikzcd}
\]
commutes. In \cite[\S 4]{tho-alex} an explicit interpretation of \eqref{paarungoben} in terms of qfh-sheaves is given. This motivates the following construction of the pairing \eqref{pairing}.

\medskip\noindent
Let $A$ be an abelian group and $A\to \I^\bullet$ an injective
resolution of the constant sheaf $A$ in the category of qfh-sheaves. An element of $H^1_{\rW,t}(X,A)$ is represented
by a pair
\[
(\alpha,\beta)\in \I^1(\bar X)\oplus \I^0(\bar X),
\] with $d\alpha=0$, $[\alpha] \in H^1_t(\bar X,A)$ and $d\beta= \alpha-F^*\alpha$.
An element in $H_1^\WS(X,\Z)$ is represented by a pair
\[
(x,y) \in \Cor(\bar \Delta^1,\bar X) \oplus \Cor(\bar \Delta^0,\bar X),
\]
with $x\partial= F^{-1}yF -y  $.

\medskip
Since $H^1_t(\bar \Delta^1,A)=0= H^1_\et(\bar \Delta^0,A)$,  we can find $s\in  \I^0(\bar \Delta^1)$ with $ds=x^*F^*\alpha \in \I^1(\bar \Delta^1)$ and $t\in \I^0(\bar \Delta^0)$ with $dt=y^*\alpha \in \I^1(\bar\Delta^0)$.
Then
\begin{equation}
\langle(x,y),(\alpha,\beta)\rangle:= F^* t-t-\partial^*s + y^*\beta
\end{equation}
lies in
\[
A=H^0_\et(\bar \Delta^0,A)= \ker( \I^0(\bar \Delta^0) \stackrel{d}{\longrightarrow} \I^1(\bar \Delta^0)).
\]
Indeed, we have
\begin{multline*}
d(F^* t-t-\partial^*s + y^*\beta)=  F^*y^*\alpha - y^*\alpha - \partial^*x^*F^*\alpha  + y^*(\alpha-F^*\alpha)\\
= F^*y^*\alpha - (F^*y^*{{F^*}}^{-1}-y^*)(F^*\alpha)-y^*F^*\alpha=0.
\end{multline*}
One checks without difficulty that $\langle(x,y),(\alpha,\beta)\rangle$ does not depend on the choices of $s$ and $t$.

\begin{lemma}\label{pairingcorrect}
$\langle(x,y),(\alpha,\beta)\rangle\in A$ only depends on the class of\/ $(x,y)$ in $H_1^\WS(X,\Z)$ and on the class of\/ $(\alpha,\beta)$ in $H^1_{\rW,t}(X,A)$.
\end{lemma}
\begin{proof}
For $\gamma\in \I^0(\bar X)$,  $s=x^*F^*\gamma$ and $t= y^*\gamma$ satisfy the condition and we have
\begin{multline*}
\langle(x,y),(d\gamma,\gamma-F^*\gamma)\rangle= F^* t-t-\partial^*s + y^*(\gamma-F^*\gamma)\\
= F^*y^*\gamma-y^*\gamma - (x\delta)^*F^*\gamma= F^*y^*\gamma-y^*\gamma - F^*y^*\gamma + y^*F^*\gamma=0.
\end{multline*}
Let $(u,v) \in \Cor(\bar \Delta^2,\bar X) \oplus \Cor(\bar \Delta^1,\bar X)$. Since $H^1_t(\bar \Delta^2,A)=0=H^1_t(\bar \Delta^1,A)$, we can find $\sigma\in \I^0(\bar \Delta^2)$ with $d\sigma = u^*F^*\alpha$ and $\tau\in \I^0(\bar\Delta^1)$ with $d\tau= -v^*\alpha$. Then $s=\partial^*\sigma+F^*\tau -\tau - v^*\beta$ and $t=\partial^*\tau$ satisfy
\[
ds= (u\partial+v-F^{-1}vF) F^*\alpha, \quad dt=(-v\partial)^*\alpha
\]
 and
\begin{multline*}
\langle(u\partial + v-F^{-1}vF,-v\partial  ), (\alpha,\beta)\rangle= F^*t-t-\partial^*s - (v\partial)^* \beta=\\
F^*\partial^*\tau-\partial^*\tau - \partial^*(\partial+\sigma^*F^*\tau-\tau-v^*\beta) -\partial^*v^*\beta =0.
\end{multline*}
\end{proof}

By Lemma~\ref{pairingcorrect}, we obtain the pairing \eqref{pairing} and the reciprocity map \eqref{rec}.

\begin{lemma}
If $f: X\to Y$  is an $\F$-morphism, then the diagram
\[
\begin{tikzcd}[column sep=tiny]
H_1^\WS(X,\Z)\dar[swap]{f_*} &{\times} & H^1_{\rW,t}(X,A)\ \arrow{rrrrr}{\langle\ ,\ \rangle_X} &&&&&\phantom{}& A\dar[equal]\\
H_1^\WS(Y,\Z) &{\times} & H^1_{\rW,t}(Y,A)\ \uar{f^*} \arrow{rrrrr}{\langle\ ,\ \rangle_Y} &&&&&\phantom{}& A
\end{tikzcd}
\]
commutes, hence $\rec_X: H_1^\WS(X,\Z) \to \Pi_1^{t,\ab}(X)_\rW$ is functorial in $X$.
\end{lemma}
\begin{proof}
Let $A\to \I^\bullet$ an injective resolution of the constant sheaf $A$ in the category of qfh-sheaves. Let $c\in H^1_{\rW,t}(Y,A)$ be represented
by
\[
(\alpha,\beta)\in \I^1(\bar Y)\oplus \I^0(\bar Y),
\]
with $d\alpha=0$, $[\alpha] \in H^1_t(\bar Y,A)$ and $d\beta= \alpha-F^*\alpha$. Furthermore, let $\zeta \in H_1^\WS(X,\Z)$ be represented by
\[
(x,y) \in \Cor(\bar \Delta^1,\bar X) \oplus \Cor(\bar \Delta^0,\bar X),
\]
with $x\partial= F^{-1}yF -y  $.
We have to show that
$
\langle \zeta , f^*(c)\rangle_X = \langle f_*(\zeta), c \rangle_Y
$.
This follows directly from the construction: First note that
$f_*(\zeta)$ is represented by $(f\circ x, f\circ y)$.
Choose $s\in  \I^0(\bar \Delta^1)$ with
\[
ds=(f\circ x)^*F^*\alpha \in \I^1(\bar \Delta^1)
\]
and $t\in \I^0(\bar \Delta^0)$ with $dt=(f\circ y)^*\alpha \in \I^1(\bar\Delta^0)$.
Then
\begin{equation*}
\langle (f\circ x, f\circ y),(\alpha,\beta)\rangle_Y= F^* t-t-\partial^*s + (f\circ y)^*\beta
\end{equation*}
by definition. Since $ds= x^*F^*(f^*(\alpha))$, $dt= y^*(f^*(\alpha))$,
we obtain
\[
\begin{array}{rcl}
\langle (x,y), f^*(\alpha,\beta)\rangle_X&=& F^* t-t-\partial^*s + y^*(f^*(\beta))\\
&=& F^* t-t-\partial^*s + (f\circ y)^*\beta\\
&=&\langle (f\circ x, f\circ y),(\alpha,\beta)\rangle_Y,
\end{array}
\]
showing the assertion.
\end{proof}

\begin{proposition}\label{pointcase}
The composite of\/ $\rec_X$ with the natural map $H_0^\rS(X,\Z) \to\break H_1^\WS(X,\Z)$ is the map
\[
r_X: H_0^\rS(X,\Z)  \longrightarrow \Pi_1^{t,\ab}(X)_\rW
\]
which sends the class $[x]$ of a closed point $x\in X$ to its Frobenius automorphism in $\Pi_1^{t,\ab}(X)_\rW$.
\end{proposition}
\begin{proof} By functoriality, it suffices to consider the case $X=\Delta^0$.  In this case,
we have natural identifications
$\Z=H_0^\rS(\Delta^0,\Z)= H_1^\WS(\Delta^0,\Z)$ (sending $1\in \Z$ to  $\id_{\Delta^0}\in \Cor(\Delta^0,\Delta^0)$), and for any abelian group $A$, we have $
A= H^0_\et(\bar \Delta^0,A)_G = H_{\rW,t}^1(\Delta^0,A)$.
With respect to these identifications, the pairing~\eqref{pairing} is just multiplication $\Z \times A\to A$, $(n,a)\mapsto na$. Furthermore, the isomorphism of Proposition~\ref{homvergleich}
\[
A= H_{\rW,t}^1(\Delta^0,A) \stackrel{\sim}{\longrightarrow} \Hom(\Pi_1^{t,\ab}(\Delta^0)_\rW,A)=  \Hom(G,A)
\]
maps $a\in A$ to the homomorphism $G\to A$, which sends the Frobenius $F\in G\cong\Z$ to $a$. Using all this, the statement of the proposition follows from the definition of the reciprocity map.
\end{proof}

\section{Comparison of blow-up sequences}

\medskip
If $A=\Z/m$, then Weil-etale and etale cohomology agree, and by Corollary~\ref{sushomdivcor}, the pairing \eqref{pairing} induces a pairing
\begin{equation}
\label{pairingmodm}
H_1^\WS(X,\Z/m) \times H^1_{t}(X,\Z/m) \longrightarrow \Z/m,
\end{equation}
and hence a map
\begin{equation}
\Phi^1_X: H_1^\WS(X,\Z/m) \longrightarrow H^1_{t}(X,\Z/m)^*,
\end{equation}
which is the $\bmod$ $m$-version of $\rec_X$.

\medskip
In addition, we consider the pairing
\begin{equation}
H_0^\WS(X,\Z/m) \times H^0_{\et}(X,\Z/m) \longrightarrow \Z/m
\end{equation}
defined as follows: Choose $x\in \Cor(\bar \Delta^0, \bar X)$ representing a class in $H_0^\WS(X,\Z/m)\allowbreak = H_0^\rS(\bar X,\Z/m)_G$ and $\alpha \in I^0(\bar X)$ with $d\alpha=0$ and $\alpha -F^*\alpha=0$. Then put $\langle x,\alpha\rangle=x^*\alpha \in H^0_\et(\bar \Delta^0,\Z/m) = \Z/m$.
We obtain a map
\begin{equation}
\Phi_X^0: H_0^\WS(X,\Z/m) \longrightarrow H^0_\et(X,\Z/m)^*.
\end{equation}
The maps $\Phi^0$ and $\Phi^1$ extend in a natural way to non-connected schemes. They induce a map from the exact sequence of Proposition \ref{blowupseq} to the dual of the exact sequence of Proposition~\ref{blowcohom}. The
compatibility with the boundary map is given by

\begin{proposition} \label{anticommuting}
For any abstract blow-up square
\[
\begin{tikzcd}
Z'\rar{i'}\dar[swap]{\pi'}&X'\dar{\pi}\\
Z\rar{i}& X
\end{tikzcd}
\]
the following diagram is commutative
\[
\begin{tikzcd}
H_1^\WS(X,\Z/m)\rar{\delta}\dar[swap]{\Phi_X^1}& H_{0}^\WS(Z',\Z/m)\dar{\Phi_{Z'}^0}\\
H^1_t(X,\Z/m)^*\rar{-\delta^*} &\ H^{0}_\et(Z',\Z/m)^*.
\end{tikzcd}
\]
Here $\delta$ is the boundary map of the exact sequence of Proposition \ref{blowupseq} and $\delta^*$ is the dual of the boundary map of the exact sequence of Proposition~\ref{blowcohom}.
\end{proposition}

\begin{proof}
We have to show that the diagram
\[
\begin{tikzcd}[column sep=tiny]
H_1^\WS(X,\Z/m)\dar[swap]{\delta} &{\times} & H^1_{t}(X,\Z/m)\ \arrow{rrrrr}{\langle\ ,\ \rangle} &&&&&\phantom{}&\Z/m\dar[equal]\\
H_0^\WS(Z',\Z/m) &{\times} & H^0_{\et}(Z',\Z/m)\ \uar{-\delta} \arrow{rrrrr}{\langle\ ,\ \rangle} &&&&&\phantom{}& \Z/m
\end{tikzcd}
\]
commutes.  Let $a\in H_1^\WS(X,\Z/m)$ and $b\in H^0_\et(Z',\Z/m)$. We put
\[
C_i^\rW(X)=C_i(\bar X) \otimes \Z/m \, \oplus \, C_{i-1}(\bar X)\otimes \Z/m,
\]
i.e., $C_\bullet^\rW(X)$ is the complex calculating $H_\bullet^{\WS}(X,\Z/m)$. Consider the diagram
\[
\begin{tikzcd}[column sep=large, row sep=large]
0\rar& C_1^\rW(Z')\rar{(i'_*,-\pi'_*)} \dar{(\partial^*,1-F^*)}&C_1^\rW(X') \oplus C_1^\rW(Z)\rar{(\pi_* , i_*)} \dar{(\partial^*,1-F^*)}& C_1^\rW(X)\dar{(\partial^*,1-F^*)}\\
0\rar& C_0^\rW(Z')\rar{(i'_*,-\pi'_*)}&C_0^\rW(X') \oplus C_0^\rW(Z)\rar{(\pi_* , i_*)} & C_0^\rW(X) .
\end{tikzcd}
\]
By Proposition~\ref{blowupseq} and its proof,  $a\in H_1^\WS(X,\Z/m)$ can be represented by a cocycle $\alpha\in C_1^\rW(X)$ which can be lifted to
$C_1^\rW(X') \oplus C_1^\rW(Z)$.
We choose $\hat \alpha  \in C_1^\rW(X') \oplus C_1^\rW(Z)$ with $(\pi_*,i_*)(\hat \alpha)=\alpha$, hence
$(\pi_*,i_*)(\partial^*,1-F^*)(\hat \alpha)=0$.  We conclude that
 $\delta(a)\in H_0^\WS(Z',\Z/m)$ is represented by an element $\gamma\in C_0^\rW(Z')$ with
\begin{equation}
(i'_*,-\pi'_*)(\gamma)= (\partial^*,1-F^*)(\hat \alpha).
\end{equation}
Let $\I^\bullet$  be an injective resolution of $\Z/m$ in the category of sheaves of
$\Z/m$-modules on the $h$-site on $\Sch/\F$, and let $\beta \in \I^{0}(Z')$, $d\beta=0$,
be a representative of $b\in H^{0}_\et(Z',\Z/m)$.
Consider the diagram
\[
\begin{tikzcd}[column sep=small]
\phantom{x}&&\I^0(X')\oplus \I^0(Z)\arrow[bend right]{lldd}[swap]{\widehat\alpha^*} \dar{d}\rar{(i'^*,-\pi'^*)}&\I^0(Z')\dar{d}\\
&\I^1(X) \arrow[hookrightarrow]{r}{(\pi^*,i^*)}\dar{\alpha^*}& \I^1(X')\oplus \I^1(Z)\rar{(i'^*,-\pi'^*)} \dar{\widehat\alpha^*} &\I^1(Z')\\
\I^0(\bar \Delta^1)\oplus \I^0(\bar \Delta^0)\rar{d}\dar{(\partial^*, 1-F^*)}&\I^1(\bar \Delta^1)\oplus \I^1(\bar \Delta^0) \dar{(\partial^*, 1-F^*)}\arrow[dash,yshift=0.3ex]{r}\arrow[dash,yshift=-0.3ex]{r} &\I^1(\bar \Delta^1)\oplus \I^1(\bar \Delta^0) \dar{(\partial^*, 1-F^*)}\rar{d}&\I^2(\bar \Delta^1)\oplus \I^2(\bar \Delta^0) \dar{(\partial^*, 1-F^*)}\\
\I^0(\bar \Delta^0)\rar{d}&\I^1(\bar\Delta^0)\arrow[dash,yshift=0.3ex]{r}\arrow[dash,yshift=-0.3ex]{r}&\I^1(\bar \Delta^0)\rar{d}&\I^2(\bar \Delta^0).
\end{tikzcd}
\]
Since the complex $\coker\big(\I^\bullet(X')\oplus \I^\bullet(Z)\to \I^\bullet(Z')\big)$ is exact (cf.\ the exact triangle \eqref{eqtriang} in the proof of Proposition~\ref{blowcohom}), we find $\hat \beta\in \I^0(X')\oplus \I^0(Z)$ with $({i'}^*,-{\pi'}^*)(\hat\beta)=\beta$.
By the argument of \cite[Lem.\,12.7]{MVW}, the sequence
\[
0\to \f(X) \to \f(X')\oplus \f(Z) \to \f(Z')
\]
is exact for every $h$-sheaf $\f$. Therefore the second line in the diagram is exact and there  is a unique $\varepsilon \in \I^1(X)$ with $(\pi^*,i^*)(\varepsilon)=d\hat \beta$ representing $\delta(b)\in H^1_t(X,\Z/m)$.
From $\hat \alpha^*(d\hat \beta)=\alpha^*(d\varepsilon)$ it follows that
\[
d(\hat\alpha^*(\hat \beta))=\alpha^*(\varepsilon) \in \ker(\partial^*, 1-F^*).
\]
By definition, we have
\[
\langle a, \delta b\rangle = -(\partial^*,1-F^*)\hat\alpha^*(\hat \beta) \in \ker(\I^0(\bar \Delta^0)\to \I^1(\bar \Delta^0))=H^0_\et(\bar\Delta^0,\Z/m).
\]
On the other hand, $\langle \delta a, b \rangle = \gamma^*(b)\in H^0_\et(\bar \Delta^0,\Z/m)$ is represented by $\gamma^*\beta\in \I^0(\Delta^0)$ and
\[
\gamma^*\beta=\gamma^*({i'}^*,-{\pi'}^*)(\hat \beta)= \big(i'_*(\gamma),-\pi'_*(\gamma)\big)^*(\hat \beta)=\big( (\partial^*,1-F^*)(\hat\alpha)\big)^*(\hat \beta).
\]
Now we write $\hat \alpha=(\hat \alpha_1,\hat \alpha_2)$ with $\hat \alpha_1 \in C_1(\bar X')\oplus C_1(\bar Z)$ and $\hat \alpha_2 \in C_0(\bar X')\oplus C_0(\bar Z)$.

Then $(\partial^*,1-F^*)(\hat\alpha)= \alpha_1 \partial + \alpha_2 - F^{-1}\alpha_2F$. Since $F^* (\hat \beta)=\hat \beta$, we conclude
\[
\big( (\partial^*,1-F^*)(\hat\alpha)\big)^*(\hat \beta)= (\delta^*,1-F^*)\hat \alpha^*(\hat \beta).
\]
This completes the proof.
\end{proof}

\section{Proof of the main theorem}
To prove our main theorem, we first consider finite coefficients.

\begin{proposition}
The map $\Phi_X^0$ is an isomorphism for any $X$ and $m$.
\end{proposition}

\begin{proof} We can assume that $X$ is connected. Then, by Proposition~\ref{sushomdiv}, the degree map induces an isomorphism $H_0^\WS(X,\Z/m)\stackrel{\sim}{\to} H_0^\WS(\F,\Z/m)\cong\Z/m$. Furthermore, $H^0_\et(\F,\Z/m) \stackrel{\sim}{\to} H_\et^0(X,\Z/m)$. Hence, by functoriality, we can reduce the statement to the case $X=\Spec(\F)$, where it is easy.
\end{proof}

\begin{theorem} \label{mainthmmodm}
For any separated scheme of finite type over a finite field $\F$, the map
\[ \Phi_X^1: H_1^{\WS}(X,\Z/m) \longrightarrow H^1_t(X,\Z/m)^* \]
is surjective. It is an isomorphism if $m$ is prime to
the characteristic or if resolutions of singularities holds for schemes of dimension $\le \dim X+1$ over $\F$.
\end{theorem}

\begin{proof}
By Propositions \ref{blowcohom}, \ref{blowupseq} and \ref{anticommuting}, and induction on the dimension, we can assume
that $X$ is normal and connected. Then, by Proposition~\ref{pointcase} and Chebotarev-Lang density, the composite
\[
 H_0^\rS(X,\Z)/m \longrightarrow H_1^{\WS}(X,\Z/m)\stackrel{\Phi_X^1}{\longrightarrow} H^1_t(X,\Z/m)^*\]
is surjective, hence so is $\Phi_X^1$.
To get the isomorphism, we note that by Theorem~\ref{duall}, the
source and the target of $\Phi_X^1$ have the same order under the given hypothesis.
\end{proof}

\begin{proof}[Proof of Theorem~\ref{maintheorem}]
Consider the diagram (for any $m$)
\[ \begin{tikzcd}
H_1^{\WS}(X,\Z) \dar[swap]{\varphi}\rar{\rec_X}&  \pw(X)_\rW \dar{\psi}\\
H_1^{\WS}(X,\Z/m )\rar{\Phi_X^1}& \pw(X)_\rW/m.
\end{tikzcd}
\]
The composite $\Phi_X^1\circ \varphi$ is surjective by Corollary~\ref{sushomdivcor} and Theorem~\ref{mainthmmodm}. Hence the cokernel of $\rec_X$ is divisible.
Since  $\pw(X)_\rW$  is finitely generated, all divisible elements of $H_1^{\WS}(X,\Z)$ are in
the kernel of $\rec_X$, and the cokernel of $\rec_X$ is
finitely generated and divisible, hence trivial.

Now assume that resolution of singularities holds for schemes of dimension $\le \dim X+1$ over $\F$.
Then $\Phi_X^1$ is an isomorphism for all $m$. Hence the kernel of $\rec_X$ is the set $\div H_1^\WS(X,\Z)$ of  divisible elements. This agrees with the
maximal divisible subgroup by the following Lemma~\ref{divlemma}.
\end{proof}

\begin{lemma}\label{divlemma} Let $A$ be an abelian group.
If $A/\div A$ is finitely generated, then $\div A$ is divisible.
\end{lemma}

\begin{proof}
Let $B= A/\div A$, choose an integer $n$ such that $nB$ is free, and
let $C\subseteq A$ be the inverse image of $nB$ in $A$.
Then $\div C=\div A$ because $nA\subseteq C\subseteq A$. By
freeness of $nB$, we obtain that $C=nB \oplus \div A$, hence
$\div A=\div C = \div \div A$.
\end{proof}

\section{The case of proper curves}

In this section we illustrate the results of this paper in the case of proper curves.

\medskip\noindent
Let $C$ be a connected, proper curve over $\F$. The morphism $C_\red \to C$ induces an isomorphism on Weil-Suslin homology as well as on fundamental groups. We therefore may assume that $C$ is reduced.  We first note that
\[
H_1^\WS(C,\Z)
\]
is a finitely generated abelian group.
This follows by applying Proposition~\ref{blowupseq} to the normalization morphism $\widetilde C \to C$ and by using Examples~\ref{wsofpoint} and \ref{wssmothcurve}. Furthermore, resolution of singularities holds for schemes of dimension $\le 2$ over $\F$. Hence Theorem~\ref{maintheorem} yields the reciprocity \emph{isomorphism}
\[
\rec_C:  H_1^\WS(C,\Z) \stackrel{\sim}{\longrightarrow} \Pi_1^\ab(C)_\rW.
\]
Since $C$ is proper, we have $\CH_0(C)=H_0^\rS(C,\Z)$, hence \eqref{sustows} yields an injection
\[
\phi: \CH_0(C) \hookrightarrow H_1^\WS(C,\Z).
\]
By \cite[Thm.\,6.2 and Prop.\,6.3]{tho-arithm}, $\coker(\phi)$ is isomorphic to $H_1(\Gamma,\Z)$,  where $\Gamma$ is the dual graph associated with the curve $C$. We obtain an exact diagram
\[
\begin{tikzcd}
0\rar&\CH_0(C)\dar[equal]\rar{\phi} & H_1^\WS(C,\Z)\rar\dar{\rec_C}[swap]{\wr} & H_1(\Gamma,\Z)\rar\dar[equal]&0\phantom{,}\\
0\rar&\CH_0(C)\rar{r_C} &  \Pi_1^\ab(C)_\rW \rar& H_1(\Gamma,\Z)\rar&0,
\end{tikzcd}
\]
where, by Proposition~\ref{pointcase}, $r_C$ is the map which sends the class of a closed point to its Frobenius automorphism.  Denoting by $\widehat{A}$ the profinite completion of an abelian group $A$,  completion of the lower line yields the exact sequence
\begin{equation}\label{kssequence}
0 \longrightarrow \widehat{\CH_0(C)} \longrightarrow \pi_1^\ab(C) \longrightarrow \widehat{H_1(\Gamma,\Z)} \longrightarrow 0,
\end{equation}
which is the exact sequence of \cite[Prop.\,1]{katosaito}.

\bigskip \noindent
Finally, we consider the example of a nodal curve.  Let $\widetilde{C}$ be a  smooth, proper curve over $\F$ admitting two rational points $P,Q$, and let $C$ be the curve obtained by identifying $P$ and $Q$. We denote the image of $P$ and $Q$ in $C$ by $O$, i.e., we have an (abstract) blow-up square
\begin{equation} \label{nodeblow}
\begin{tikzcd}
\{P,Q\} \rar\dar & \widetilde{C}\dar\\
\{O\}\rar &C.
\end{tikzcd}
\end{equation}

Let $D$ be a countable chain of copies $C_i$, $i\in \Z$, of $C$ with  $P_i \in C_i$ identified with $Q_{i-1}\in C_{i-1}$ for all $i\in \Z$:
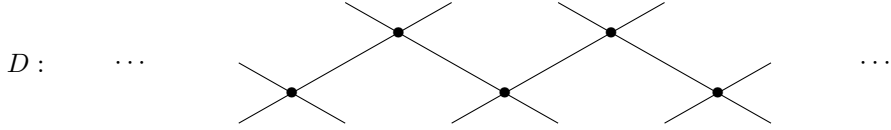
\[
\begin{tikzpicture}[xscale=.7,yscale=.4]
\node at (-1,2) {$D:$};
\node at (1,2) {$\cdots$};
\node at (4,1) {$\bullet$};
\node at (6,3) {$\bullet$};
\node at (8,1) {$\bullet$};
\node at (10,3) {$\bullet$};
\node at (12,1) {$\bullet$};
\draw (5,0) -- (3,2);
\draw (3,0) --(7,4);
\draw (5,4) --(9,0);
\draw (7,0) --(11,4);
\draw (9,4) --(13,0);
\draw (11,0) --(13,2);
\node at (15,2) {$\cdots$};
\end{tikzpicture}
\]
Translation gives a natural $\Z$-action on $D$ and $C$ is the quotient $D/\Z$ with respect to this action. In particular, we have a surjection $\Pi_1^\ab(C) \twoheadrightarrow \Aut_C(D)=\Z$.
Inspecting the combinatorics of $D$,  we see that
\[
\ker(\Pi_1^\ab(C) \to \Z)\cong \Pi_1^\ab(\widetilde{C}) / (\Frob_P-\Frob_Q).
\]
Hence, we have an exact sequence
\begin{equation}\label{pionenode}
0 \longrightarrow \Pi_1^\ab(\widetilde{C})_W/(\Frob_P-\Frob_Q) \longrightarrow \Pi_1^\ab(C)_W \longrightarrow \Z \longrightarrow 0.
\end{equation}
In particular,
\[
\Pi_1^\ab(C)_W \cong \Z \oplus \Z \oplus \textrm{(finite)}
\]
as an abelian group. Applying Proposition~\ref{blowupseq} to the blow-up square \eqref{nodeblow} and using Examples~\ref{wsofpoint} and \ref{wssmothcurve}, we obtain the exact sequence
\begin{equation}\label{wsnode}
0 \longrightarrow \CH_0(\widetilde{C})/ ([P]-[Q]) \longrightarrow H_1^\WS(C,\Z) \longrightarrow \Z \longrightarrow 0.
\end{equation}
The reciprocity map $\rec_C$ induces an isomorphism between the exact sequences \eqref{wsnode}
and \eqref{pionenode}. The map on the left hand side is induced by the reciprocity map of the smooth curve $\widetilde{C}$, which sends $[x]\in \CH_0(\widetilde{C})$ to $\Frob_x \in \Pi_1^\ab(\widetilde{C})_W=\pi_1^\ab(\widetilde{C})_W$.

\vskip1cm
\small

{\sc  Rikkyo University, Department of Mathematics, 3-34-1 Nishi-Ikebukuro, Toshima-ku,
Tokyo Japan 171-8501}

\textit{E-mail address:} {\tt geisser@rikkyo.ac.jp}

\medskip
{\sc  Universit\"{a}t Heidelberg, Mathematisches Institut, Im Neuenheimer Feld 205, D-69120 Heidelberg, Deutschland}

\textit{E-mail address:} {\tt schmidt@mathi.uni-heidelberg.de}
\end{document}